\documentclass[final]{siamltex}
\usepackage{amssymb}
\usepackage{amsmath}
\usepackage{amsfonts}
\usepackage{algorithm}
\usepackage{graphicx}
\usepackage{color}
\usepackage{epsfig}

\numberwithin{algorithm}{section}
\newtheorem{remark}{Remark}
\newtheorem{example}{Example}

\newcommand{\bu}{\mathbf u}

\begin{document}

\title{A positivity preserving iterative method for finding the ground states of saturable nonlinear Schr\"odinger equations}
\author{Ching-Sung Liu }

\maketitle

\begin{abstract}
In this paper, we propose an iterative method to compute the positive ground states  of saturable nonlinear Schr\"odinger equations.  A discretization of the saturable nonlinear Schr\"odinger equation leads to a nonlinear algebraic eigenvalue problem (NAEP).  For any initial positive vector, we prove that this method converges globally with a locally quadratic convergence rate to a positive solution of NAEP.  During the iteration process, the method requires the selection of a positive parameter $\theta_k$ in the $k$th iteration, and generates a positive vector sequence approximating the eigenvector of NAEP and a scalar sequence approximating the corresponding eigenvalue. We also present a halving procedure to determine the parameters $\theta_k$, starting with $\theta_k=1$ for each iteration,  such that the scalar sequence is strictly monotonic increasing. This method can thus be used to illustrate the existence of positive ground states of saturable nonlinear Schr\"odinger equations. Numerical experiments are provided to support the theoretical results.

\end{abstract}




\pagestyle{myheadings} \thispagestyle{plain}


\section{Introduction}
The nonlinear Schr\"odinger (NLS) equation \cite{R92} is a nonlinear variation of the Schr\"odinger equation and is a general model in nonlinear science and mathematics. Such an equation can be expressed as follows:
\begin{equation}
i\frac{\partial \phi}{\partial z}+\triangle \phi+\Gamma  f(|\phi |^2)\phi=0 \text{ for some constant } \Gamma \in  \mathbb{R},\label{NLS}
\end{equation}
where $\phi = \phi(x, z): \mathbb{R}^2 \times \mathbb{R}^{+} \rightarrow \mathbb{C}$,  the function $f$ denotes the nonlinearity and $i$ is the imaginary unit.
A NLS equation is called a saturable NLS equation \cite{C82,L17} if the nonlinear function $f(s)=1-1/(a+s^2)$, that is, 
\begin{equation}
i\frac{\partial \phi}{\partial z}=-\triangle \phi+\Gamma \left (1-\frac{1}{a(x)+| \phi |^2}\right )\phi , \text{ for } \Gamma >0, \label{NLS1}
\end{equation}
where $a(x) >0$ is a bounded function. A saturable NLS equation is of interest  in several applications \cite{G97,K92,K65,M13,M68,M07}, and has been extensively studied in the past thirty years. In many application areas, one is interested in finding the ground state vector of equation (\ref{NLS1}). The ground state of equation (\ref{NLS1}) is defined as the minimizer of the energy function, which is determined by the following constrained optimization problem \cite{C82,L17}:
\begin{equation}
m(\Gamma,I)=\inf \{ H(u)\text{ }|\text{ } u \in H^1(\mathbb{R}^2), \int_{\mathbb{R}^2}u^2=1 \}, \label{GroundS}
\end{equation}
where
\begin{equation*}
H(u)=\int_{\mathbb{R}^2}|\nabla u|^2+\Gamma\left[ u^2- \ln\left( 1+\frac{u^2}{a(x)}\right) \right]dx.
\end{equation*}
Therefore, the associated Euler-Lagrange equation of (\ref{GroundS}) is as follows:
\begin{equation}\label{eq:NSLE}
-\Delta u+\Gamma \left (1-\frac{1}{a(x)+u^2}\right ) u=\lambda u,
\end{equation}
where $a(x)）> 0, \int_{-\infty}^{\infty}u^2(x)dx=1$, $(\lambda, u)$ is the eigenpair. In general, the eigenfunction $u(x)$ describes the probability distribution of finding a particle in a particular region in space. Therefore, the existence of positive solutions $u(x)$ \cite{L17} and the problem of computing these solutions has attracted much attention in recent years. Here we consider the finite-difference discretization of the nonlinear eigenvalue problem (\ref{eq:NSLE}) with Dirichlet boundary conditions, and the discretization gives a nonlinear algebraic eigenvalue problem (NAEP)
\begin{equation}\label{dnep}
A\mathbf{u}+\Gamma \mathrm{diag} \left (\mathbf{e}-\frac{\mathbf{e}}{\mathbf{%
a}+\mathbf{u}^{[2]}}\right ) \mathbf{u}=\lambda \mathbf{u}, \quad \mathbf{u}%
^T\mathbf{u}=1, \end{equation}
where $\mathbf{a}>0, \Gamma >0$, $\mathbf{u}=[u_{1},u_{2},\ldots ,u_{n}]^{T}\in \mathbb{R}^{n},$ $\mathbf{u}^{[2]}=[u_{1}^{2},u_{2}^{2},\ldots ,u_{n}^{2}]^{T},$ $A$ is an irreducible nonsingular $M$-matrix and $\mathbf{e}=[1,\ldots ,1]^{T}$.
We aim to provide a structure-preserving algorithm with fast convergence rate for computing positive eigenvectors $\mathbf{u}_{\ast}$ and eigenvalues $\lambda_{\ast}$of NAEP (\ref{dnep}) and giving a detailed convergence analysis.

  In many applications, the positivity structure of the approximate solutions is important; if the approximations lose positivity structure, then they may be meaningless and unexplained. Therefore, in this paper, we propose a positivity preserving iteration for nonlinear algebraic eigenvalue problems (\ref{dnep})  by combining the idea of Newton's method with the idea of the Noda iteration \cite{N71}, called the Newton-Noda iteration (NNI). NNI is a Newton iterative method with a new type of full Newton steps, it has the advantage that no line-searches are needed, and naturally preserves the strict positivity of the target eigenvector $\mathbf{u}_{\ast}$ in its approximations at all iterations.  We also present a halving procedure to determine the parameters $\theta_k$, starting with $\theta_k=1$ for each iteration,  such that the sequence approximating target eigenvalue $\lambda_{\ast}$ is strictly monotonic increasing and bounded, and thus its global convergence is guaranteed. Another advantage of NNI is that it converges quadratically and computes the desired eigenpair correctly for any positive initial vector.

The rest of this paper is organized as follows. In Section 2, we present a Newton-Noda iteration. In Section 3, we  prove some basic properties for Newton-Noda iteration. Section 4 addresses the global convergence and the local convergence rate of NNI. In Section 5, we provide numerical examples to verify the theoretical results and the performance of NNI.
Some concluding remarks are given in the last section.

Throughout this paper, we use the bold face letters to denote a vector and use the $2$-norm for vectors and matrices. The superscript $T$ denotes the transpose of a vector or matrix, and we use $\mathbf{v}^{(i)}$ to represent the $i$th element of a vector $\mathbf{v} $. $\mathbf{v}^{[m]}$ denotes element-by-element powers, i.e., $\mathbf{v}^{[m]}=[v_{1}^{m},v_{2}^{m},\ldots ,v_{n}^{m}]^{T}.$ A real matrix $A=\left[ A_{ij}\right] \in \mathbb{R}^{n\times k}$ is called nonnegative (positive) if $A_{ij}\geq 0$ $(A_{ij}>0)$. For real matrices $A$ and $B$ of the same size, we write $A\geq B$ ($A>B$) if $A-B$ is nonnegative (positive). A real square matrix $A$ is called a $Z$-matrix if all its off-diagonal elements are nonpositive. A matrix $A$ is called a M-matrix if it is a Z-matrix with $A^{-1} \geq 0$. A matrix $A$ is called reducible \cite{BPl94,HJo85} if
there exists a nonempty proper index subset $S\subset \left\{
1,2,\ldots ,n\right\} $ such that%
\begin{equation*}
A_{ij}=0,\text{ }\forall \  i\in S,\text{ }\forall
\  j\notin S.
\end{equation*}%
If $A$ is not reducible, then we call $A$ irreducible.
For a pair of positive vectors $\mathbf{v}$ and $\mathbf{w}$, define
\begin{equation*}
\max \left( \frac{\mathbf{w}}{\mathbf{v}}\right) =\underset{i}{\max }\left(
\frac{\mathbf{w}^{(i)}}{\mathbf{v}^{(i)}}\right) ,\text{ \ }\min \left(
\frac{\mathbf{w}}{\mathbf{v}}\right) =\underset{i}{\min }\left( \frac{%
\mathbf{w}^{(i)}}{\mathbf{v}^{(i)}}\right) .
\end{equation*}

\section{The Newton-Noda iteration}

In this section, we will present a Newton-Noda iteration (NNI) for computing a positive eigenvector $\mathbf{u}_{\ast}$ of NAEP (\ref{dnep}), and then we prove some basic properties of NNI in Section 3, which will be used to establish its convergence theory in Section 4.

First,  NAEP (\ref{dnep}) can be simplified as follows:
\begin{equation*}
\mathcal{A}(\mathbf{u})\mathbf{u}=\lambda \mathbf{u},
\end{equation*}
where
\begin{equation*}
\mathcal{A}(\mathbf{u})=A+\Gamma \mathrm{diag} \left (\mathbf{e}-\frac{\mathbf{e}}{%
\mathbf{a}+\mathbf{u}^{[2]}}\right )
\end{equation*}
and $\mathrm{diag}\left ( \ast \right )$ returns a square diagonal matrix with the elements of vector $\ast$
 on the main diagonal. 
We define two vector-valued functions $\mathbf{r}:$ $\mathbb{R}_{+}^{n+1}%
\mathbb{\rightarrow R}^{n}$ and 
$F:$ $\mathbb{R}_{+}^{n+1}\mathbb{\rightarrow R}^{n+1}$ as follows:
\begin{equation}
\mathbf{r}(\mathbf{u,}\lambda )=\mathcal{A}(\mathbf{u})\mathbf{u}-\lambda \mathbf{u}, \quad \mathbf{F}(\mathbf{u},\lambda )=\left[
\begin{array}{c}
-\mathbf{r}(\mathbf{u,}\lambda ) \\
\frac{1}{2}\left( 1-\mathbf{u}^{T}\mathbf{u}\right)%
\end{array}%
\right].   \label{eq:Fx}
\end{equation}
The Fr\'echet derivative of $F$ is given by
\begin{equation}\label{eq:Fre}
F^{\prime}(\mathbf{u},\lambda)=\left[
\begin{array}{cc}
J({\mathbf{u}}) & -\mathbf{u} \\
-\mathbf{u}^{T} & 0%
\end{array}
\right ],
\end{equation}
where
\begin{equation*}
J(\mathbf{u})=A+(\Gamma-\lambda)I-\Gamma \mathrm{diag} \left (\frac{\mathbf{a%
}-\mathbf{u}^{[2]}}{(\mathbf{a}+\mathbf{u}^{[2]})^{[2]}}\right ).
\end{equation*}

Next, we consider using Newton's method to solve the equation $\mathbf{F}(\mathbf{u,}\lambda )=0$. Given an approximation $(\mathbf{u}_k,\widehat{\lambda }_{k})$, Newton's method produces the next approximation $(\mathbf{u}_{k+1},\widehat{\lambda }_{k+1})$ as follows:

\begin{align}
\left[
\begin{array}{cc}
J(\mathbf{u}_k) & -\mathbf{u}_k \\
-\mathbf{u}_k^{T} & 0%
\end{array}%
\right] \left[
\begin{array}{c}
\Delta_{k} \\
\delta _{k}%
\end{array}%
\right]&=-\left[
\begin{array}{c}
\mathbf{r}(\mathbf{u}_k,\widehat{\lambda }_{k}) \\
\frac{1}{2}\left( \mathbf{u}_k^{T}\mathbf{u}_k-1\right)%
\end{array}%
\right],  \label{eq:step1} \\
\mathbf{u}_{k+1}& =\mathbf{u}_k\,+\Delta_{k},
\label{eq:step2} \\
\widehat{\lambda }_{k+1}& =\widehat{\lambda }_{k}+\delta _{k}.
\label{eq:step3}
\end{align}
From the first equation of (\ref{eq:step1}), we have
\begin{eqnarray*}
J(\mathbf{u}_k)(\Delta_k+\mathbf{u}_k) &=&J(\mathbf{u}_k)\Delta_k+J(\mathbf{u}_k)\mathbf{u}_k\\
&=&\delta_k \mathbf{u}_k-\mathbf{r}(\mathbf{u}_k,\widehat{\lambda }_{k})+ J(\mathbf{u}_k)\mathbf{u}_k \\
&=&\delta_k \mathbf{u}_k-(\mathcal{A}(\mathbf{u}_k)-\widehat{\lambda }_{k} I)\mathbf{u}_k\\&+& (\mathcal{A}(\mathbf{u}_k)-\widehat{\lambda }_{k}  I)\mathbf{u}_k+2\Gamma \mathrm{diag} \left (\frac{%
\mathbf{u}_k^{[2]}}{(\mathbf{a}+\mathbf{u}_k^{[2]})^{[2]}}\right )\mathbf{u}_k \\
&=& \delta_k \mathbf{u}_k+2\Gamma \mathrm{diag} \left (\frac{%
\mathbf{u}_k^{[2]}}{(\mathbf{a}+\mathbf{u}_k^{[2]})^{[2]}}\right )\mathbf{u}_k.
\end{eqnarray*}%
Hence, 
\begin{equation}
\mathbf{u}_{k+1}=J(\mathbf{u}_k)^{-1}\left (\delta_k \mathbf{u}_k+2\Gamma \mathrm{diag} \left (\frac{%
\mathbf{u}_k^{[2]}}{(\mathbf{a}+\mathbf{u}_k^{[2]})^{[2]}}\right )\mathbf{u}_k\right ).
\label{eq: linearsys0}
\end{equation}%

Since $\mathbf{u}_k$ is going to approximate the positive eigenvector of NAEP, we will also require $\mathbf{u}_k>0$. However,
we cannot guarantee $\mathbf{u}_{k+1}>0$ in \eqref{eq: linearsys0}
unless we have 
\begin{equation*}
\delta_k>0, J(\mathbf{u}_k)^{-1}\geq0.
\end{equation*}
What is needed here is that $J(\mathbf{u}_k)$ is a nonsingular M-matrix.
For $\mathbf{u}_k>0$, we suggests taking
\begin{equation}\label{eq:lamk}
\lambda_k =\min \left( \frac{\mathcal{A}(\mathbf{u}_k)\mathbf{u}_k}{\mathbf{u}_k}\right),
\end{equation}
which is precisely the idea of the Noda iteration \cite{N71}.
This implies that the $Z$-matrix $\mathcal{A}(\mathbf{u}_k)-\lambda_k I$ is such that
$(\mathcal{A}(\mathbf{u}_k)-\lambda_k I)\mathbf{u}_k\ge 0$. Thus $\mathcal{A}(\mathbf{u}_k)-\lambda_k I$
is a nonsingular $M$-matrix when $(\mathbf{u}_k, \lambda_k)$ is not an
eigenpair, and is a singular $M$-matrix when $(\mathbf{u}_k, \lambda_k)$ is an
eigenpair. 
Since
\begin{equation}\label{JandA}
J(\mathbf{u}_k)-(\mathcal{A}(\mathbf{u}_k)-\lambda_k I)=2\Gamma \mathrm{diag} \left (\frac{%
\mathbf{u}_k^{[2]}}{(\mathbf{a}+\mathbf{u}_k^{[2]})^{[2]}}\right ),
\end{equation}
we have $J(\mathbf{u}_k)\mathbf{u}_k>0$. Thus $J(\mathbf{u}_k)$ is a nonsingular $M$-matrix. 
Based on (\ref{eq:step1}), (\ref{eq:step2}) and (\ref{eq:lamk}), we can present NNI as Algorithm \ref{alg1}.

\begin{algorithm}
\begin{enumerate}
  \item   Given $\bu_0 > 0$ with $\Vert \bu_0\Vert =1$, ${\lambda}_{0}= \min \left( \frac{\mathcal{A}(\mathbf{u}_{0}){\bf u}_0}{\mathbf{u}_{0}}\right)$ and tol $>0$.
  \item   {\bf for} $k =0,1,2,\dots$
  \item   \quad Solve the linear system $F'(\mathbf{u_k},\lambda_k ) \left [\begin{array}{c}
\Delta_k\\
\delta_k
\end{array}\right ]=-   F(\mathbf{u_k},\lambda_k )        $.
  \item \quad Choose a scalar $\theta_k>0$.
  \item   \quad Compute the vector ${\mathbf{w}}_{k+1} =\mathbf{u}_{k}\,+\theta_k \Delta_k$.
   \item   \quad Normalize the vector $\bu_{k+1}= {\mathbf{w}}_{k+1}/\Vert {\mathbf{w}}_{k+1}\Vert$.
  \item   \quad Compute  ${\lambda }_{k+1} =\min \left( \frac{\mathcal{A}(\mathbf{u}_{k+1}){\bf u}_{k+1}}{\mathbf{u}_{k+1}}\right)$.
  \item   {\bf until} convergence: $\Vert \mathcal{A}(\mathbf{u}_k)-\underline{\lambda}_k\mathbf{u}_k\Vert < $tol.
\end{enumerate}
\caption{Newton-Noda iteration (NNI)}
\label{alg1}
\end{algorithm}
In what follows, we will prove the positivity of $\mathbf{u}_{k}$ and give a strategy for choosing $\theta _{k}$. These results will show that Algorithm~\ref{alg1} is a positivity preserving algorithm.

\subsection{Positivity of $\mathbf{u}_{k}$}
Suppose that $\left\{\mathbf{u}_{k}, {\lambda }_{k}\right\}$ is generated by
Algorithm \ref{alg1}. We now prove that the parameter $\theta_k \in (0,1]$ in Algorithm \ref{alg1} naturally preserves
the strict positivity of $\mathbf{u}_{k}$ at all iterations.

For any vector $\mathbf{u}>0$, from (\ref{eq:Fre}), it follows that
\begin{equation}\label{eq:Fnon}
F^{\prime}(\mathbf{u}, \lambda)=\left[
\begin{array}{cc}
I & 0 \\
-\mathbf{u}^{T} (J(\mathbf{u}))^{-1} & 1%
\end{array}
\right ] \left[
\begin{array}{cc}
J({\mathbf{u}}) & -\mathbf{u} \\
0 & -\mathbf{u}^{T} (J(\mathbf{u}))^{-1} \mathbf{u}%
\end{array}
\right ]
\end{equation}
is nonsingular and
\begin{align}  \label{eq1.1}
(F^{\prime}(\mathbf{u}, \lambda))^{-1}=\left[
\begin{array}{cc}
(J(\mathbf{u}))^{-1} -\frac{(J(\mathbf{u}))^{-1} \mathbf{u} \mathbf{u}^T(J(%
\mathbf{u}))^{-1}}{\mathbf{u}^{T}(J(\mathbf{u}))^{-1} \mathbf{u}} & -\frac{
(J(\mathbf{u}))^{-1} \mathbf{u} }{\mathbf{u}^{T} (J(\mathbf{u}))^{-1}
\mathbf{u}} \\
-\frac{ \mathbf{u}^T (J(\mathbf{u}))^{-1} }{\mathbf{u}^{T} (J(\mathbf{u}%
))^{-1} \mathbf{u}} & -\frac{ 1 }{\mathbf{u}^{T} (J(\mathbf{u}))^{-1}
\mathbf{u}}%
\end{array}
\right ].
\end{align}

\begin{lemma}
\label{lem1} Given $\mathbf{a},\Gamma >0$. Suppose that $\lambda_k\in \mathbb{R}$ and $\mathbf{u}_k>0$
with $\|\mathbf{u}_k\|=1$ such that $\mathbf{r}(\mathbf{u}_k,\lambda_k )\geqslant 0$. Then
\begin{align}  \label{eq1.2}
\frac{1}{2}(1+ \mathbf{u}_k^T \mathbf{u}_k)-\Gamma \mathbf{u}_k^T (J(\mathbf{%
u}_k))^{-1} \mathrm{diag} \left (\frac{2\mathbf{u}_k^{[2]}}{(\mathbf{a}+%
\mathbf{u}_k^{[2]})^{[2]}}\right )\mathbf{u}_k\geqslant 0.
\end{align}
Moreover, the equality holds if and only if $\mathbf{r}(\mathbf{u}_k,\lambda_k )= 0$.
\end{lemma}

\begin{proof}
Since $\mathcal{A}({\bf u}_k){\bf u}_k-\lambda_k{\bf u}_k\geqslant 0$, $J({\bf u}_k)=(\mathcal{A}({\bf u}_k)-\lambda_k I)+\Gamma  {\rm diag} \left (\frac{2{\bf u}_k^{[2]}}{({\bf a}+{\bf u}_k^{[2]})^{[2]}}\right )$ is nonsingular $M$-matrix. Then we have
\begin{align*}
\Gamma(J({\bf u}_k))^{-1}{\rm diag} \left (\frac{2{\bf u}_k^{[2]}}{({\bf a}+{\bf u}_k^{[2]})^{[2]}}\right ){\bf u}_k=\left(I-(J({\bf u}_k))^{-1}(\mathcal{A}({\bf u}_k)-\lambda_k I)\right){\bf u}_k.
\end{align*}
Hence,
\begin{align*}
\frac{1}{2}&(1+ {\bf u}_k^T {\bf u}_k)-\Gamma {\bf u}_k^T (J({\bf u}_k))^{-1}  {\rm diag} \left (\frac{2{\bf u}_k^{[2]}}{({\bf a}+{\bf u}_k^{[2]})^{[2]}}\right ){\bf u}_k\\
&=\frac{1}{2}(1+ {\bf u}_k^T {\bf u}_k)-{\bf u}_k^T {\bf u}_k+{\bf u}_k^T (J({\bf u}_k))^{-1}(\mathcal{A}({\bf u}_k){\bf u}_k-\lambda_k {\bf u}_k)\\
&=\frac{1}{2}(1- {\bf u}_k^T {\bf u}_k)+{\bf u}_k^T (J({\bf u}_k))^{-1}\mathbf{r}(\mathbf{u}_k,\lambda_k ).
\end{align*}
Since $\|{\bf u}_k\|=1$ and $\mathbf{r}(\mathbf{u}_k,\lambda_k )\geqslant 0$, \eqref{eq1.2} holds by using ${\bf u}_k^T (J({\bf u}_k))^{-1}>0$. It is easily seen that the equality of \eqref{eq1.2} holds if and only if $\mathbf{r}(\mathbf{u}_k,\lambda_k )=0$.
\end{proof}

\begin{theorem}\label{positivethm}
Given $\mathbf{a},\Gamma >0$. Assume $\left\{\mathbf{u}_{k}, {\lambda }_{k}\right\}$ is
generated by Algorithm \ref{alg1}. If $\theta_k \in (0,1]$, then $\mathbf{u}_{k}>0$ for all $k$.
\end{theorem}

\begin{proof}
Since ${\bf u}_0>0$, by mathematical induction, it suffices to show that if ${\bf u}_k>0$ then ${\bf u}_{k+1}>0$. Suppose that ${\bf u}_k>0$, it follows from the step 3 of Algorithm  \ref{alg1} that
\begin{align*}
F'(\mathbf{u_k},\lambda_k ) \left [\begin{array}{c}
{\bf u}_k+\Delta_k\\
\delta_k
\end{array}\right ]&=-\left[
\begin{array}{c}
\mathcal{A}({\bf u}_k){\bf u}_k-\lambda_k{\bf u}_k  \\
\frac{1}{2}\left( 1-{\bf u}_k^{T}{\bf u}_k\right )
\end{array}
\right] +\left [\begin{array}{c}
 J({\bf u}_k){\bf u}_k\\
-{\bf u}_k^{T}{\bf u}_k
\end{array}\right ]\\
&=\left [\begin{array}{c}
\Gamma  {\rm diag} \left (\frac{2{\bf u}_k^{[2]}}{({\bf a}+{\bf u}_k^{[2]})^{[2]}}\right ){\bf u}_k\\
-\frac{1}{2}\left( 1+{\bf u}_k^{T}{\bf u}_k\right )
\end{array}\right ].
\end{align*}
By \eqref{eq1.1}, we have
\begin{align}\label{eq:udelta}
{\bf u}_k+\Delta_k =&\Gamma\left(I-\frac{(J({\bf u}_k))^{-1} {\bf u}_k {\bf u}_k^T}{\mathbf{u}_k^{T}(J({\bf u}_k))^{-1} {\bf u}_k} \right) (J({\bf u}_k))^{-1} \notag {\rm diag} \left (\frac{2{\bf u}_k^{[2]}}{({\bf a}+{\bf u}_k^{[2]})^{[2]}}\right ){\bf u}_k\\ \notag
&+\frac{1+ {\bf u}_k^T {\bf u}_k}{2\mathbf{u}_k^{T}(J({\bf u}_k))^{-1} {\bf u}_k} (J({\bf u}_k))^{-1}{\bf u}_k\\ \notag
=&\Gamma(J({\bf u}_k))^{-1}  {\rm diag} \left (\frac{2{\bf u}_k^{[2]}}{({\bf a}+{\bf u}_k^{[2]})^{[2]}}\right ){\bf u}_k\\ 
&+\frac{\frac{1}{2}(1+ {\bf u}_k^T {\bf u}_k)-\Gamma {\bf u}_k^T (J({\bf u}_k))^{-1}  {\rm diag} \left (\frac{2{\bf u}_k^{[2]}}{({\bf a}+{\bf u}_k^{[2]})^{[2]}}\right ){\bf u}_k}{\mathbf{u}_k^{T}(J({\bf u}_k))^{-1} {\bf u}_k} (J({\bf u}_k))^{-1}{\bf u}_k.
\end{align}
Since $\Gamma>0$, ${\bf u}_k>0$, $J({\bf u}_k)$ is a nonsingular $M$-matrix and $\mathbf{u}_k^{T}(J({\bf u}_k))^{-1} {\bf u}_k>0$, it follows from Lemma \ref{lem1} that ${\bf u}_k+\Delta_k>0$.
Therefore, ${\bf w}_{k+1}={\bf u}_k+\theta_k \Delta_k>0$ if $0< \theta_k \le 1 $, and hence,  $\bu_{k+1}= {\mathbf{w}}_{k+1}/\Vert {\mathbf{w}}_{k+1}\Vert>0$.
\end{proof}

\begin{remark}
$\text{ }$
\begin{enumerate}
\item[(i)] $\mathbf{u}_k^T\Delta_k=0$: From \eqref{eq:udelta}, it is easily seen that $\mathbf{u}_k^T\Delta_k=\frac{1}{2}(1-\mathbf{u}_k^T \mathbf{u}_k)=0.$

\item[(ii)] $\delta _{k}\geqslant 0$: From the step 3 of Algorithm \ref{alg1}
and using \eqref{eq1.1}, we have
\begin{align*}
\delta _{k}& =\frac{1}{\mathbf{u}_{k}^{T}(J(\mathbf{u}_{k}))^{-1}\mathbf{u}%
_{k}}\mathbf{u}_{k}^{T}(J(\mathbf{u}_{k}))^{-1}\mathbf{r}(\mathbf{u}_k,\lambda_k )+\frac{\frac{1}{2}(1-\mathbf{u}_{k}^{T}%
\mathbf{u}_{k})}{\mathbf{u}_{k}^{T}(J(\mathbf{u}_{k}))^{-1}\mathbf{u}_{k}} \\
& =\frac{1}{\mathbf{u}_{k}^{T}(J(\mathbf{u}_{k}))^{-1}\mathbf{u}_{k}}\mathbf{%
u}_{k}^{T}(J(\mathbf{u}_{k}))^{-1}\mathbf{r}(\mathbf{u}_k,\lambda_k )\geqslant 0.
\end{align*}%
\end{enumerate}
\end{remark}

\begin{lemma}
\label{equiThm}If $\delta_k$, $\Delta_k$ and $\mathbf{u}_{k}$ are generated by
Algorithm~\ref{alg1}, then the following statements are equivalent:%
\begin{equation*}
\text{{\rm (i) }}\delta _{k}=0;\quad \text{ {\rm (ii) }}\mathbf{r}(\mathbf{u}_k,\lambda_k )=0; \quad \text{{\rm (iii) }}\Delta_k =0.
\end{equation*}
\end{lemma}
\begin{proof}
From the step 3 of Algorithm~\ref{alg1}, we have
\begin{equation*}
F'(\mathbf{u_k},\lambda_k ) \left [\begin{array}{c}
\Delta_k\\
\delta_k
\end{array}\right ]=-\left[
\begin{array}{c}
\mathbf{r}(\mathbf{u}_k,\lambda_k )  \\
0
\end{array}
\right].
\end{equation*}

(i)$\Rightarrow $(ii): From (ii) of Remark 1, we get $\delta _{k}=0$ if and only if $\mathbf{r}(\mathbf{u}_k,\lambda_k )=0$.

(ii)$\Rightarrow $(iii): Since $\mathbf{r}(\mathbf{u}_k,\lambda_k )=0$ and $F'(\mathbf{u_k},\lambda_k )$ is a nonsingular matrix, we have $\Delta_k=0$ and $\delta_k=0$.

(iii) $\Rightarrow $ (i): 
If $\Delta_k =0$, then
\begin{equation*}
F'(\mathbf{u_k},\lambda_k ) \left [\begin{array}{c}
0\\
\delta_k
\end{array}\right ]=-\left[
\begin{array}{c}
\mathbf{r}(\mathbf{u}_k,\lambda_k )  \\
0
\end{array}
\right],
\end{equation*}
and it follows
$$-\delta_k \mathbf{u}_k=-\mathbf{r}(\mathbf{u}_k,\lambda_k )=-\left(\mathcal{A}({\bf u}_k){\bf u}_k-\lambda_k{\bf u}_k\right),$$
which implies
$$\mathcal{A}({\bf u}_k){\bf u}_k=\left(\lambda_k+\delta_k\right){\bf u}_k.$$
Then 
$${\lambda }_{k} +\delta_k=\min \left( \frac{\mathcal{A}(\mathbf{u}_{k}){\bf u}_{k}}{\mathbf{u}_{k}}\right)=\lambda_k,$$
which means $\delta_k=0$. 
\end{proof}

\subsection{The strategy for choosing $\theta_k$}
In this section, we would like to choose $\theta_k \in (0,1]$ such that the sequence $\left\{ \lambda _{k}\right\} $ is strictly increasing and bounded above. 

\begin{lemma}
\label{posiyk}
Given
a unit vector $\mathbf{u}_{k}>0$ and $\theta_k \in (0,1]$, then  
\begin{equation}
\lambda _{k+1}=\lambda _{k}+\min \left( \frac{\mathbf{h%
}_{k}(\theta _{k})}{\mathbf{u}_{k+1}}\right),
\label{eq: recurrLam}
\end{equation}%
where $\mathbf{h}_{k}(\theta_k )=\mathbf{r}(\mathbf{u}_{k+1},\lambda _{k})$.
Moreover, $\mathbf{h}_{k}(\theta _{k})$ can be also expressed in the form
\begin{equation}
\mathbf{h}_{k}(\theta_k )=\frac{1-\theta_k}{\Vert {\mathbf{w}}_{k+1}\Vert}\mathbf{r}(\mathbf{u}_k,\lambda_k)+\frac{\theta_k \delta_k}{\Vert {\mathbf{w}}_{k+1}\Vert} \mathbf{u}_k+\mathbf{R}(\theta_k \Delta_k) \label{eq:ftheta}
\end{equation}
where $\Vert\mathbf{R}(\theta_k \Delta_k)\Vert \leq M\Vert \theta_k \Delta_k\Vert^2$.
\end{lemma}

\begin{proof}
By Theorem~\ref{positivethm}, we know that $\mathbf{u}_{k}>0$ for all $k$.
\begin{equation*}
\lambda _{k+1}-\lambda _{k}=\min \left( \frac{\mathcal{A}(\mathbf{u}_{k+1}){\bf u}_{k+1}}{\mathbf{u}_{k+1}}\right)-\lambda _{k}=\min \left( \frac{\mathbf{h}_{k}(\theta _{k})}{\mathbf{u}_{k+1}}\right),
\end{equation*}%
where $\mathbf{h}_{k}(\theta_k )=\mathbf{r}(\mathbf{u}_{k+1},\lambda _{k})$.
By Taylor’s theorem, we have 
\begin{eqnarray}
\mathbf{r}(\mathbf{u}_{k+1},\lambda_k)&=&\mathbf{r}(\mathbf{u}_{k},\lambda_k)+J(\mathbf{u}_{k})(\mathbf{u}_{k+1}-\mathbf{u}_{k})+\mathbf{E}_k  \notag\\
&=&\mathbf{r}(\mathbf{u}_{k},\lambda_k)+J(\mathbf{u}_{k})(\frac{\mathbf{u}_{k}+\theta_k \Delta_k}{\Vert {\mathbf{w}}_{k+1}\Vert}-\mathbf{u}_{k})+\mathbf{E}_k  \notag\\
&=&\mathbf{r}(\mathbf{u}_{k},\lambda_k)+\left(\frac{1}{\Vert {\mathbf{w}}_{k+1}\Vert}-1\right)J(\mathbf{u}_{k})\mathbf{u}_{k}+\frac{\theta_k }{\Vert {\mathbf{w}}_{k+1}\Vert}J(\mathbf{u}_{k})\Delta_k+\mathbf{E}_k\notag\\
&=&\mathbf{r}(\mathbf{u}_{k},\lambda_k)+\left(\frac{1}{\Vert {\mathbf{w}}_{k+1}\Vert}-1\right)\left[\mathbf{r}(\mathbf{u}_{k},\lambda_k)+2\Gamma \mathrm{diag} \left (\frac{%
\mathbf{u}_k^{[2]}}{(\mathbf{a}+\mathbf{u}_k^{[2]})^{[2]}}\right )\mathbf{u}_{k}\right] \notag\\
&+&\frac{\theta_k }{\Vert {\mathbf{w}}_{k+1}\Vert}\left[\delta_k\mathbf{u}_{k}- \mathbf{r}(\mathbf{u}_{k},\lambda_k)\right]+\mathbf{E}_k \notag\\
&=& \frac{1-\theta_k}{\Vert {\mathbf{w}}_{k+1}\Vert}\mathbf{r}(\mathbf{u}_k,\lambda_k)+\frac{\theta_k \delta_k\mathbf{u}_k}{\Vert {\mathbf{w}}_{k+1}\Vert} 
+\frac{1-\Vert {\mathbf{w}}_{k+1}\Vert}{\Vert {\mathbf{w}}_{k+1}\Vert}\left[\frac{2\Gamma\mathbf{u}_k^{[3]}}{(\mathbf{a}+\mathbf{u}_k^{[2]})^{[2]}} \right]+\mathbf{E}_k, \label{eq:ruk}
\end{eqnarray}
where $\Vert\mathbf{E}_k\Vert \leq M_1\Vert \mathbf{u}_{k+1}-\mathbf{u}_{k}\Vert^2$.

Since $\mathbf{u}_{k}^{T}\Delta_k=0$ from Remark 1, we have $\Vert {\mathbf{w}}_{k+1}\Vert= \sqrt{1+\Vert \theta_k \Delta_k\Vert^2 }.$
Hence, the third term in the right-hand side of (\ref{eq:ruk}) is bounded by
%
\begin{eqnarray}
\Vert \frac{1-\Vert {\mathbf{w}}_{k+1}\Vert}{\Vert {\mathbf{w}}_{k+1}\Vert} \left[\frac{2\Gamma\mathbf{u}_k^{[3]}}{(\mathbf{a}+\mathbf{u}_k^{[2]})^{[2]}} \right]\Vert
&\le &\Vert \frac{1-\sqrt{1+\Vert \theta_k \Delta_k\Vert^2 }}{\sqrt{1+\Vert \theta_k \Delta_k\Vert^2 }} \Vert \Vert\frac{2\Gamma\mathbf{u}_k^{[3]}}{\mathbf{a}^{[2]}+2\mathbf{a}\mathbf{u}_k^{[2]}+\mathbf{u}_k^{[4]}} \Vert  \notag \\
&\leq& \frac{\Gamma}{2\min{(\mathbf{a})}}\Vert \theta_k \Delta_k\Vert^2, \label{eq:wk1}
\end{eqnarray}
and the upper bound of $\Vert\mathbf{E}_k\Vert $ can be re-estimated as follows:
\begin{eqnarray}
\Vert\mathbf{E}_k\Vert 
&\leq &M_1\Vert \mathbf{u}_{k+1}-\mathbf{u}_{k}\Vert^2 \notag \\
&=&M_1\Vert \left(\frac{1}{\Vert {\mathbf{w}}_{k+1}\Vert}-1\right)\mathbf{u}_{k}+\frac{\theta_k }{\Vert {\mathbf{w}}_{k+1}\Vert}\Delta_k\Vert^2 \notag \\
&\leq &M_1\Vert \left(\frac{1}{\Vert {\mathbf{w}}_{k+1}\Vert}-1\right)\mathbf{u}_{k}\Vert^2+\frac{M_1 }{\Vert {\mathbf{w}}_{k+1}\Vert}\Vert \theta_k\Delta_k\Vert^2  \notag \\
&\leq & \left(\frac{M_1}{2}+M_1\right)\Vert \theta_k \Delta_k\Vert^2. \label{eq:resi}
\end{eqnarray}
From the above relation (\ref{eq:ruk})-(\ref{eq:resi}), we have
\begin{equation*}
\mathbf{h}_{k}(\theta_k )=\frac{1-\theta_k}{\Vert {\mathbf{w}}_{k+1}\Vert}\mathbf{r}(\mathbf{u}_k,\lambda_k)+\frac{\theta_k \delta_k}{\Vert {\mathbf{w}}_{k+1}\Vert} \mathbf{u}_k+\mathbf{R}(\theta_k \Delta_k),
\end{equation*}
where
\begin{equation*}
\mathbf{R}(\theta_k \Delta_k)=\frac{1-\Vert {\mathbf{w}}_{k+1}\Vert}{\Vert {\mathbf{w}}_{k+1}\Vert}\left[\frac{2\Gamma\mathbf{u}_k^{[3]}}{(\mathbf{a}+\mathbf{u}_k^{[2]})^{[2]}} \right]+\mathbf{E}_k
\end{equation*}
with $\Vert\mathbf{R}(\theta_k \Delta_k)\Vert \leq M\Vert \theta_k \Delta_k\Vert^2$ and $M=\frac{\Gamma}{2 \min(\mathbf{a})}+\frac{M_1}{2}+M_1$.
\end{proof}

We next show that $\lambda _{k}$ is strictly increasing and bounded above for suitable $\theta_k$, unless  $\mathbf{u}_{k}$ is an eigenvector of NAEP for some $k$, in which case NNI terminates with $\lambda _{k}$.

\begin{theorem}
\label{monotone}Suppose $A$ be an irreducible M-matrix and $%
\eta >0$ be a fixed constant. Given a unit vector $\mathbf{u}_{k}>0
$, suppose $\mathbf{u}_{k}\not=\mathbf{u}_{\ast}$ and $\theta _{k}$ in Algorithm \ref{alg1} satisfies
\begin{equation}
\theta _{k}=\left\{
\begin{array}{cl}
1 & \text{if }\mathbf{h}_{k}(1)\geq \frac{\delta_k\mathbf{u}_{k}}{(1+\eta)\Vert {\mathbf{w}}_{k+1}\Vert }\text{;} \\
\eta _{k} & \text{otherwise,}%
\end{array}%
\right.  \label{eq:lowbdtheta}
\end{equation}%
where for each $k$ with $\mathbf{h}_{k}(1)< \frac{\delta_k\mathbf{u}_{k}}{(1+\eta)\Vert {\mathbf{w}}_{k+1}\Vert }$,
$$
\eta_{k}=\frac{\eta \delta_k \min \left( \mathbf{u}_{k}\right)}{(1+\eta )M\Vert {\mathbf{w}}_{k+1}\Vert \left\Vert \Delta_{k}\right\Vert^2 }. $$
Then $0<\eta_k< 1$ whenever it is defined, and
\begin{equation}
\lambda _{k}<\lambda _{k+1}<\Vert A\Vert+(1+n)\Gamma .
\label{eq:monolam}
\end{equation}

\end{theorem}

\begin{proof}
By Lemma \ref{posiyk}, we have
\begin{equation*}
\lambda _{k+1}=\lambda _{k}+\min \left( \frac{\mathbf{h%
}_{k}(\theta _{k})}{\mathbf{u}_{k+1}}\right) .
\end{equation*}%
We need to prove $\mathbf{h}_{k}(\theta _{k})>0.$

From (\ref{eq:ftheta}) and $\Vert\mathbf{R}(\theta_k \Delta_k)\Vert \leq M\Vert \theta_k \Delta_k\Vert^2$, we have
\begin{eqnarray}
\mathbf{h}_{k}(\theta )&=&\frac{\theta \delta_k\mathbf{u}_{k}}{(1+\eta) \Vert {\mathbf{w}}_{k+1}\Vert}+\frac{\theta \eta \delta_k\mathbf{u}_{k}}{(1+\eta) \Vert {\mathbf{w}}_{k+1}\Vert} +\frac{1-\theta}{\Vert {\mathbf{w}}_{k+1}\Vert}\mathbf{r}(\mathbf{u}_k,\lambda_k)+\mathbf{R}(\theta \Delta_k)  \notag \\
&>&\frac{\theta \delta_k\mathbf{u}_{k}}{(1+\eta) \Vert {\mathbf{w}}_{k+1}\Vert}+\frac{\theta \eta \delta_k\mathbf{u}_{k}}{(1+\eta) \Vert {\mathbf{w}}_{k+1}\Vert} -M\theta^2 \left\Vert \Delta_{k}\right\Vert^2 \mathbf{e}. \label{eq:ftheta3}
\end{eqnarray}%

If $\eta_k \ge 1$, then $$\eta \delta_k \min \left( \mathbf{u}_{k}\right)\ge (1+\eta )M\Vert {\mathbf{w}}_{k+1}\Vert \left\Vert \Delta_{k}\right\Vert^2, $$ and it follows $$\frac{\eta \delta_k\mathbf{u}_{k}}{(1+\eta) \Vert {\mathbf{w}}_{k+1}\Vert} \ge M \left\Vert \Delta_{k}\right\Vert^2 \mathbf{e}.$$
Thus
\begin{eqnarray}
\mathbf{h}_{k}(1)&=&\frac{ \delta_k\mathbf{u}_{k}}{(1+\eta) \Vert {\mathbf{w}}_{k+1}\Vert}+\frac{ \eta \delta_k\mathbf{u}_{k}}{(1+\eta) \Vert {\mathbf{w}}_{k+1}\Vert} +\mathbf{R}(\Delta_k)  \notag \\
&>&\frac{\delta_k\mathbf{u}_{k}}{(1+\eta)\Vert {\mathbf{w}}_{k+1}\Vert }+\frac{ \eta \delta_k\mathbf{u}_{k}}{(1+\eta)\Vert {\mathbf{w}}_{k+1}\Vert } -M \left\Vert \Delta_{k}\right\Vert^2 \mathbf{e}>0.
\end{eqnarray}
If $\eta_k < 1$, we have
\begin{equation}
\theta _{k}=\eta _{k}=\frac{\eta \delta_k \min \left( \mathbf{u}_{k}\right)}{(1+\eta )M\Vert {\mathbf{w}}_{k+1}\Vert \left\Vert \Delta_{k}\right\Vert^2 },
\label{eq: betheta1}
\end{equation}%
which ensures the inequality
\begin{equation}
\frac{\theta_k\eta \delta_k\mathbf{u}_{k}}{(1+\eta)\Vert {\mathbf{w}}_{k+1}\Vert } \ge \theta^2_k M \left\Vert \Delta_{k}\right\Vert^2 \mathbf{e}.  \label{eq: dklow}
\end{equation}%
Substituting (\ref{eq: dklow})  into (\ref{eq:ftheta3}), we
obtain
\begin{eqnarray}
\mathbf{h}_{k}(\theta _{k}) &=&\frac{\theta_k \delta_k\mathbf{u}_{k}}{(1+\eta)\Vert {\mathbf{w}}_{k+1}\Vert }+\frac{\theta_k \eta \delta_k\mathbf{u}_{k}}{(1+\eta)\Vert {\mathbf{w}}_{k+1}\Vert } +\frac{1-\theta_k}{\Vert {\mathbf{w}}_{k+1}\Vert}\mathbf{r}(\mathbf{u}_k,\lambda_k)+\mathbf{R}(\theta_k \Delta_k)  \notag \\
&\ge&\frac{\theta_k \delta_k\mathbf{u}_{k}}{(1+\eta) \Vert {\mathbf{w}}_{k+1}\Vert}+\frac{\theta_k \eta \delta_k\mathbf{u}_{k}}{(1+\eta) } -M\theta_k^2 \left\Vert \Delta_{k}\right\Vert^2 \mathbf{e}  \notag \\
&\ge&\frac{\theta_k \delta_k\mathbf{u}_{k}}{(1+\eta)\Vert {\mathbf{w}}_{k+1}\Vert }>0. \label{eq333}
\end{eqnarray}%
Therefore,
\begin{equation*}
\lambda _{k+1}=\lambda _{k}+\min \left( \frac{\mathbf{h%
}_{k}(\theta _{k})}{\mathbf{u}_{k+1}}\right)>\lambda _{k} .
\end{equation*}
Next, we prove that the sequence $\left\{ \lambda _{k}\right\} $ is bounded above.
Suppose that $\left\{ \lambda _{k}\right\} $ is unbounded. This implies that  $\lambda_k \ge N>0$ for $k$ large enough. Since $A(\mathbf{u}_{k}){\bf u}_k \ge {\lambda}_{k}\mathbf{u}_{k}$, we then have
\begin{eqnarray*}
{\lambda}_{k}&\le &|\mathbf{u}_{k}^{T}A(\mathbf{u}_{k}){\bf u}_k|\\
&\le & |\mathbf{u}_{k}^{T}A{\bf u}_k| +\Gamma \left |\sum_{i=1}^{n} (1-\frac{1}{\mathbf{a}(i)+\mathbf{u}_k^{2}(i)})\mathbf{u}_k^{2}(i)\right | \\
&\le &  \Vert A\Vert + \Gamma(1 +n) <\infty,
\end{eqnarray*}
which is a contradiction.
\end{proof}


%

From  \eqref{eq:lowbdtheta}, we know that the inequality $\mathbf{h}_{k}(1)\geq \frac{\delta_k\mathbf{u}_{k}}{(1+\eta)\Vert {\mathbf{w}}_{k+1}\Vert }$ depends on the parameter $\eta$. Therefore, if $\eta$ large enough, then we can choose $\theta_k = 1$ for which $\mathbf{h}_{k}(1)>0$ holds.
By Theorem \ref{monotone}, we can indeed choose $\theta _{k} \in (0,1]$ in NNI such that the sequence $\{\lambda _{k}\}$ is strictly increasing. However, in practice, it is difficult to determine $\eta _{k}$. Therefore, we can determine $\theta _{k}$ by repeated halving technique. More precisely, for each $k$, we can take $\theta _{k}=1$ first and check whether $\mathbf{h}_{k}(1)>0$ holds. If not, then we update $\theta _{k}$ using $\theta _{k}\leftarrow \theta_{k}/2$ and check again until we get  $\theta _{k}$ for which $\mathbf{h}_{k}(1)>0$ holds. This process of repeatedly halving will be referred to as the halving procedure. As long as $\theta _{k}$  is bounded below by a positive constant, which will be mentioned in the next section. 

\section{Some basic properties of Newton-Noda iteration }
In this section, we prove a number of basic properties of NNI, which will be used to establish its convergence theory in Section 4.
\begin{lemma}
\label{yktox}Let $A$ be an irreducible M-matrix. Assume that the sequence $\left\{ \lambda _{k},\mathbf{u}_{k}, \mathbf{w}_{k}\right\} $ is generated by Algorithm~\ref%
{alg1}. For any subsequence $\left\{ \mathbf{u}_{k_{j}}\right\} \subseteq
\left\{ \mathbf{u}_{k}\right\} ,$ we have the following results:

\begin{enumerate}
\item[{\rm (i)}] If $\mathbf{u}_{k_{j}}\rightarrow \mathbf{v}$ as $j\rightarrow
\infty ,$ then $\mathbf{v}>0.$

\item[{\rm (ii)}] $\min (\mathbf{u}_{k}) \ge m$ for some positive constant $m$.

\item[{\rm (iii)}] $\Vert \mathbf{w}_k \Vert \le \frac{1}{m}$.

\end{enumerate}
\end{lemma}

\begin{proof}
(i). 
If $%
\lim_{j\rightarrow \infty }\mathbf{u}_{k_{j}}=\mathbf{v,}$ then $\mathbf{v}\ge 0$.
Let $S$ be the set of all indices $i$ such that $\lim_{j\to \infty} \mathbf{u}_{k_j}^{(i)}=\mathbf{v}^{(i)}= 0$. Since $\left\Vert \mathbf{u}_{k_{j}}\right\Vert =1$, $S$ is a proper subset of
$\{1, 2, \ldots, n\}$. Suppose $S$ is nonempty.
Then by the definition of $\lambda _{k},$
\begin{equation*}
 \lambda_{k_{j}}=\min \left( \frac{\mathcal{A}(\mathbf{u}_{k_j}){\bf u}_{k_j}}{\mathbf{u}_{k_j}}\right) \leq
\frac{\left(\mathcal{A}(\mathbf{u}_{k_j}){\bf u}_{k_j}\right)^{(i)}}{\mathbf{u}_{k_j}^{(i)}}<\infty \text{ for all }i=1, 2, \ldots, n.
\end{equation*}%
Since $\lim_{j\rightarrow \infty } \mathbf{u}_{k_j}^{(i)}= 0$ for $i\in S$,
it holds that $\lim_{j\rightarrow \infty }\left(\mathcal{A}(\mathbf{u}_{k_j}){\bf u}_{k_j}\right)^{(i)}=\left(\mathcal{A}(\mathbf{v}){\bf v}\right)^{(i)}=0$ for $i\in S$.
Thus, $\mathcal{A}(\mathbf{v})_{i,j}=0$ for all $i\in S$ and for all $j\notin S$,
which contradicts the irreducibility of $\mathcal{A}(\mathbf{v})$. Therefore, $S$ is empty and thus $\mathbf{v}>0$.

(ii). 
Suppose $\min (\mathbf{u}_{k})$ is not
bounded below by a positive constant. Then there exists a subsequence $\{k_j\}$ such that
$\lim_{j\rightarrow \infty }  \min (\mathbf{u}_{k_j}) =0$.
Since $\|\mathbf{u}_{k_j}\|=1$, we may assume that 
$\lim_{j\to \infty} \mathbf{u}_{k_{j}} =\mathbf{v}$ exists.
Then $\lim_{j\rightarrow \infty }  \min (\mathbf{u}_{k_j}) =
\min(\mathbf{v}) = 0$. This is a contradiction since  $\mathbf{v}>0$ by (i).
Therefore,  $\min (\mathbf{u}_{k})$ is
bounded below by a positive constant. That is $\min (\mathbf{u}_{k}) \ge m$ for some positive constant $m$.

(iii).
From Remark 1, we have $\mathbf{u}_{k}^{T}\mathbf{w}_{k+1}=1$  and then 
$$\Vert \mathbf{w}_{k+1} \Vert = \frac{\mathbf{u}_k^T\mathbf{w}_{k+1}}{\cos\angle (\mathbf{u}_k,\mathbf{u}_{k+1})}=\frac{1}{\cos\angle (\mathbf{u}_k,\mathbf{u}_{k+1})}.$$
Since $\mathbf{u}_k>0$ and $\mathbf{u}_{k+1}>0$ with $\Vert \mathbf{u}_k\Vert =\Vert \mathbf{u}_{k+1}\Vert =1$, we have
$$
\cos\angle (\mathbf{u_k},\mathbf{u_{k+1}}) =\mathbf{u}_k^T\mathbf{u}_{k+1}\geq \Vert \mathbf{u}_{k+1}\Vert_1 \min(\mathbf{u}_k)
>\Vert \mathbf{u}_{k+1}\Vert \min(\mathbf{u}_k)=\min(\mathbf{u}_k),
$$
where $\Vert\cdot\Vert_1$ is the vector 1-norm. Form (ii), 
\begin{equation}
\Vert \mathbf{w}_{k+1} \Vert =\frac{1}{\cos\angle (\mathbf{u}_k,\mathbf{u}_{k+1})} \le \frac{1}{\min(\mathbf{u}_k)}\le \frac{1}{m} < \infty.
\end{equation}

\end{proof}

\begin{lemma}
\label{thetak}Assume that the sequence $\left\{ \Delta_k, \delta_k, \theta_k\right\} $ is generated by Algorithm~\ref{alg1}. We have the following results:

\begin{enumerate}
\item[{\rm (i)}] There exists a constant $\beta>0$ such that $\beta \Vert \Delta_k \Vert \le \delta_k$.

\item[{\rm (ii)}] $\theta_k = 1$ if $\Vert \Delta_{k}\Vert \le \frac{\eta \beta }{(1+\eta )M } $ where $(\eta, M)$ is as in Theorem \ref{monotone}.

\item[{\rm (iii)}] $\theta_k \ge \xi$ for some positive constant $\xi$.

\end{enumerate}
\end{lemma}

\begin{proof}
(i).
From the step 3 of Algorithm~\ref{alg1}, we have
\begin{eqnarray}\label{Deltak}
\Vert \Delta_k\Vert &\le& \Vert J(\mathbf{u}_k)^{-1}\left(\delta_k \mathbf{u}_k - \mathbf{r}(\mathbf{u}_k,\lambda_k)\right) \Vert \notag \\   
&\le& \delta_k \Vert J(\mathbf{u})_k^{-1}\mathbf{u}_k \Vert  + \Vert J(\mathbf{u}_{k})^{-1}\mathbf{r}(\mathbf{u}_k,\lambda_k)\Vert .
\end{eqnarray}
Since $\Vert J(\mathbf{u})^{-1}\mathbf{u} \Vert $ a continuous function achieves its extreme values in a compact set,
it follows
$$\max_{0<m\le \min\left(\mathbf{u}\right) \le 1, \Vert \mathbf{u}\Vert =1} \left(\Vert J(\mathbf{u})^{-1}\mathbf{u} \Vert \right) < \infty.$$
Therefore, $\Vert J(\mathbf{u})^{-1}\mathbf{u} \Vert \le M_2$ for some constant $M_2$.

On the other hand, from (ii) of Remark 1, we have 
$$\left(\mathbf{u}_{k}^{T}(J(\mathbf{u}_{k}))^{-1}\mathbf{u}_{k}\right)\delta_k = \mathbf{u}_{k}^{T}(J(\mathbf{u}_{k}))^{-1}\mathbf{r}(\mathbf{u}_k,\lambda_k).$$
Since $\mathbf{u}_k>0$ and $J(\mathbf{u}_{k})^{-1}\mathbf{r}(\mathbf{u}_k,\lambda_k)>0$, by using the same proving technique of (iii) of Lemma \ref{yktox}, we have 
$$
\cos\angle (\mathbf{u_k},\frac{J(\mathbf{u}_{k})^{-1}\mathbf{r}(\mathbf{u}_k,\lambda_k)}{\Vert J(\mathbf{u}_{k})^{-1}\mathbf{r}(\mathbf{u}_k,\lambda_k)\Vert}) >\min(\mathbf{u}_k)\ge m,
$$
which implies
\begin{eqnarray*}
\Vert J(\mathbf{u}_{k})^{-1}\mathbf{r}(\mathbf{u}_k,\lambda_k)\Vert &=& \mathbf{u}_{k}^{T}J(\mathbf{u}_{k})^{-1}\mathbf{r}(\mathbf{u}_k,\lambda_k)\sec \angle (\mathbf{u_k},\frac{J(\mathbf{u}_{k})^{-1}\mathbf{r}(\mathbf{u}_k,\lambda_k)}{\Vert J(\mathbf{u}_{k})^{-1}\mathbf{r}(\mathbf{u}_k,\lambda_k)\Vert}) \\
&=& \left(\mathbf{u}_{k}^{T}J(\mathbf{u}_{k})^{-1}\mathbf{u}_{k}\right)\delta_k \sec \angle (\mathbf{u_k},\frac{J(\mathbf{u}_{k})^{-1}\mathbf{r}(\mathbf{u}_k,\lambda_k)}{\Vert J(\mathbf{u}_{k})^{-1}\mathbf{r}(\mathbf{u}_k,\lambda_k)\Vert}) \\
&\le& \frac{\delta_k}{m} \Vert J(\mathbf{u})_k^{-1}\mathbf{u}_k \Vert \le \frac{M_2}{m}\delta_k.
\end{eqnarray*}
From (\ref{Deltak}) and the above inequality, 
\begin{eqnarray*}
\Vert \Delta_k\Vert &\le& \delta_k \Vert J(\mathbf{u})_k^{-1}\mathbf{u}_k \Vert  + \Vert J(\mathbf{u}_{k})^{-1}\mathbf{r}(\mathbf{u}_k,\lambda_k)\Vert \\
&\le& \left(M_2 +\frac{M_2}{m}\right)\delta_k:= \frac{1}{\beta}\delta_k.
\end{eqnarray*}

(ii). 
If $\Vert \Delta_{k}\Vert \le \frac{\eta \beta }{(1+\eta )M } $, then
\begin{eqnarray*}
\eta_{k}&=&\frac{\eta \delta_k \min \left( \mathbf{u}_{k}\right)}{(1+\eta )M\Vert {\mathbf{w}}_{k+1}\Vert\left\Vert \Delta_{k}\right\Vert^2 } \\
&=& \frac{\eta \delta_k \min \left( \mathbf{u}_{k}\right)}{(1+\eta )M\Vert {\mathbf{w}}_{k+1}\Vert} \frac{\delta_k}{\Vert \Delta_{k} \Vert} \frac{1}{\Vert \Delta_{k} \Vert}\\
&\ge& \frac{\eta \beta m}{(1+\eta )M\Vert {\mathbf{w}}_{k+1}\Vert} \frac{1}{\Vert \Delta_{k} \Vert} \ge 1.
\end{eqnarray*}
From the proof of Theorem \ref{monotone}, $\theta_k = 1$ when $\eta_k \ge 1.$

(iii).
From (\ref{eq:lowbdtheta}), we recall that
\begin{equation*}
\theta _{k}=\left\{
\begin{array}{cl}
1 & \text{if }\mathbf{h}_{k}(1)\geq \frac{\delta_k\mathbf{u}_{k}}{(1+\eta)\Vert {\mathbf{w}}_{k+1}\Vert }\text{;} \\
\eta _{k} & \text{otherwise,}%
\end{array}%
\right.
\end{equation*}%
where $\eta_{k}=\frac{\eta \delta_k \min \left( \mathbf{u}_{k}\right)}{(1+\eta )M\Vert {\mathbf{w}}_{k+1}\Vert\left\Vert \Delta_{k}\right\Vert^2 } .$  Suppose $\theta _{k}$ is not
bounded below by $\xi >0$. Since $\mathbf{u}_{k}$ is bounded, we can find a subsequence $\{k_{j}\} $
such that $$\lim_{j\rightarrow \infty }\theta _{k_{j}}=0, \lim_{j\rightarrow \infty }\mathbf{u}_{k_{j}}=\mathbf{v}>0.$$
Note that $\mathbf{v}>0$ by Lemma \ref{yktox}.

From (\ref{eq:Fnon}), $F^{\prime}(\mathbf{u}_{k_j}, \lambda_{k_j})$ is a nonsingular matrix, and the vector $ \left [\Delta_{k_j}^{T},\delta_{k_j} \right ]^{T}$
satisfies
\begin{equation*}
 \left [\begin{array}{c}
\Delta_{k_j}\\
\delta_{k_j}
\end{array}\right ]=-F'(\mathbf{u}_{k_j},\lambda_{k_j} )^{-1}\left[
\begin{array}{c}
\mathcal{A}({\bf u}_{k_j}){\bf u}_{k_j}-\lambda_{k_j}{\bf u}_{k_j}  \\
0
\end{array}
\right].
\end{equation*}%
 Since the sequence $\left\{ \lambda_{k}\right\} $ is
 monotonically increasing and bounded above, we  have $\lim_{k\rightarrow
\infty }\lambda_{k}=\alpha $. Therefore, 
\begin{eqnarray*}
\lim_{j\rightarrow \infty }\left [\begin{array}{c}
\Delta_{k_j}\\
\delta_{k_j}
\end{array}\right ]&=&\lim_{j\rightarrow \infty
}-F'(\mathbf{u}_{k_j},\lambda_{k_j} )^{-1}\left[
\begin{array}{c}
\mathbf{r}(\mathbf{u}_{k_j},\lambda_{k_j})  \\
0
\end{array}
\right]\\
&=&-F'(\mathbf{v},\alpha)^{-1}\left[
\begin{array}{c}
\mathbf{r}(\mathbf{v},\alpha)  \\
0
\end{array}
\right]<\infty,
\end{eqnarray*}%
which means $\left\Vert \Delta_{k_j} \right\Vert $ is bounded. 
If $\eta_k$ is defined only on a finite subset of $\{k_j\}$, then $\theta_{k_j}=1$ except for a finite number of
$j$ values, contradicting $\lim_{j\rightarrow \infty }\theta _{k_{j}}=0$.
If $\eta_k$ is defined on an infinite subset $\{k_{j_i}\}$ of $\{k_j\}$,  then
\begin{eqnarray*}
0=\lim_{i\rightarrow \infty }\eta _{k_{j_i}}&=& \lim_{i\rightarrow \infty } \frac{\eta \delta_{k_{j_i}} \min \left( \mathbf{u}_{k_{j_i}}\right)}{(1+\eta )M\Vert {\mathbf{w}}_{k_{j_i}+1}\Vert\left\Vert \Delta_{k_{j_i}}\right\Vert^2 } \\
&\ge& \lim_{i\rightarrow \infty } \frac{\eta \delta_{k_{j_i}} m}{(1+\eta )Mm\left\Vert \Delta_{k_{j_i}}\right\Vert^2 }
\\
&=& \lim_{i\rightarrow \infty } \frac{\eta \delta_{k_{j_i}} }{(1+\eta )M\left\Vert \Delta_{k_{j_i}}\right\Vert } \frac{1}{\Vert \Delta_{k_{j_i}}\Vert}\\
&\ge& \lim_{i\rightarrow \infty } \frac{\eta \beta }{(1+\eta )M } \frac{1}{\Vert \Delta_{k_{j_i}}\Vert}.
\end{eqnarray*}
It follows that 
$\lim_{i\rightarrow \infty } \Vert \Delta_{k_{j_i}}\Vert = \infty.$
 This is contradictory to $\Vert \Delta_{k_j}  \Vert < \infty$.
\end{proof}

%

\section{Convergence analysis}

In this section, we prove that the convergence of NNI is global and quadratic, assuming that $\mathbf{u}_{k}\neq \mathbf{u}_{\ast }$ for each $k$.  

\subsection{Global convergence of NNI}

Theorem \ref{monotone} shows that the sequence $\left\{\lambda_{k}\right\} $ is strictly increasing and bounded above by a constant  and hence converges. We now show that the limit of $\lambda_{k}$ is precisely the eigenvalue $\lambda_{\ast}$ of NAEP (\ref{dnep}).

\begin{theorem}
\label{main}Let $A$ be an irreducible M-matrix and the sequence $\left\{ \lambda_{k}\right\} $ is
generated by Algorithm~\ref{alg1}. If $\mathbf{a}, \Gamma>0$, then the NAEP (\ref{dnep}) has a positive eigenvecor.
\end{theorem}

\begin{proof}
From (\ref{eq: recurrLam}), (\ref{eq333}) and Lemma~\ref{thetak}, we have
\begin{eqnarray}
\lambda_{k+1}-\lambda_{k}&=&\min \left( \frac{%
\mathbf{h}_{k}(\theta _{k})}{\mathbf{u}_{k+1}}%
\right) \geq \min \left(\frac{\theta_k \delta_k\mathbf{u}_{k}}{(1+\eta)\Vert {\mathbf{w}}_{k+1}\Vert \mathbf{u}_{k+1}}\right)  \notag \\
&\geq &\min \left(\frac{\xi \delta_k\mathbf{u}_{k}}{(1+\eta) \mathbf{w}_{k+1}}\right).  \label{eq: 3term}
\end{eqnarray}%

From (iii) of Lemma \ref{yktox}, we have $\|\mathbf{w}_{k+1}\|
\le \frac{1}{m} < \infty$.
It follows from (\ref{eq: 3term}) that $\lim_{k\rightarrow \infty } \delta_k \min (\mathbf{u}_{k})=0$.
From (ii) of Lemma \ref{yktox},  $\min (\mathbf{u}_{k})$ is
bounded below by a positive constant, and thus
$\lim_{k\rightarrow \infty }\delta_k=0.$

Let $\mathbf{v}$ be any limit point of $\{\mathbf{u}_k\}$, with $\lim_{j\to \infty} \mathbf{u}_{k_j}=\mathbf{v}>0$.
From Lemma~\ref{equiThm}, we then have $\lim_{j\to \infty}\delta _{k_j}=0$ if and only if $\lim_{j\to \infty} \left(\mathcal{A}(\mathbf{u}_{k_j})\mathbf{u}_{k_j}-\lambda _{k_j}\mathbf{u}_{k_j}\right)=0$, which means
$\mathcal{A}(\mathbf{v})\mathbf{v}=\lambda \mathbf{v}$. Therefore, $\mathbf{v}$ is a positive eigenvector of NAEP and $\lambda =\min \left( \frac{\mathcal{A}(\mathbf{v})\mathbf{v}}{\mathbf{v}}\right)$ is the corresponding eigenvalue, i.e.,  $\mathbf{u}_{\ast}=\mathbf{v}$ and $ \lambda_{\ast}=\lambda.$
\end{proof}

The above theorem guarantees the global convergence of NNI and also proves the existence of positive eigenvectors of NAEP.

\subsection{Quadratic convergence of NNI}

In the previous section, we discussed the global convergence of NNI. In the following section, we will establish a convergence rate analysis by exploiting a connection between NNI and Newton’s method. So we start with the following result about Newton's method.

\begin{lemma}
\label{eslonyk} Suppose that $\left( \mathbf{u}_{k}, \lambda_{k}\right) $ form Algorithm~\ref{alg1} is sufficiently close to an eigenpair $\left( \mathbf{u}_{\ast}, \lambda_{\ast}\right) $ with $\mathbf{u}_{\ast} >0$ and $\Vert\mathbf{u}_{\ast}\Vert=1$. Let $\left\{ \allowbreak \widehat{\mathbf{u}}_{k},\widehat{\lambda }_{k}\right\}$ be obtained by
Newton's method from $\left( \mathbf{u}_{k}, \lambda_{k}\right) $, i.e., 
$$\widehat{\mathbf{u}}_{k} = \mathbf{u}_{k}+\Delta_k, \text{  } \widehat{\lambda }_{k}=\lambda_k+\delta_k,$$
where $\Delta_k$ and $\delta_k$ as in Algorithm \ref{alg1}.
Then there is a constant $\beta $ such that
for all $\left( \mathbf{u}_{k}, \lambda_{k}\right) $
sufficiently close to $\left( \mathbf{u}_{\ast },    \lambda_{\ast})\right)$
\begin{equation}
\left\Vert \left[
\begin{array}{c}
\widehat{\mathbf{u}}_{k+1} \\
\widehat{\lambda }_{k+1}%
\end{array}%
\right] -\left[
\begin{array}{c}
\mathbf{u}_{\ast } \\
\lambda_{\ast}%
\end{array}%
\right] \right\Vert \leq c \left\Vert \left[
\begin{array}{c}
\mathbf{u}_{k} \\
\lambda_{k}%
\end{array}%
\right] -\left[
\begin{array}{c}
\mathbf{u}_{\ast } \\
\lambda_{\ast}%
\end{array}%
\right] \right\Vert ^{2},  \label{eq: qudraNT}
\end{equation}%
\end{lemma}

\begin{proof}
We already know that $F^{\prime}\left( \mathbf{u}_{k}, \lambda_{k}\right) $ is nonsingular. It is also clear that $F^{\prime}\left( \mathbf{u}, \lambda \right) $ satisfies a Lipschitz condition at $\left( \mathbf{u}_{\ast}, \lambda_{\ast}\right)$ since its Fr\'echet derivative is continuous in a neighborhood of $\left( \mathbf{u}_{\ast}, \lambda_{\ast}\right)$. The inequality (\ref{eq: qudraNT}) is then a basic result of Newton's method.
\end{proof}

\begin{lemma}\label{relation}
Let $\left( \mathbf{u}_{\ast}, \lambda_{\ast}\right) $ be an eigenpair with $\mathbf{u}_{\ast} >0$ and $\Vert\mathbf{u}_{\ast}\Vert=1$. Let $\left\{ \mathbf{u}_{k},   \lambda_{k}\right\}$ be generated by NNI. Then there are constants $ c_2>0$ such that
 $ \left |\lambda _{k}-\lambda_{\ast}\right |  \le
 c_2 \|\mathbf{u}_{k}-\mathbf{u}_{\ast }\|$ for all $\mathbf{u}_{k}$ sufficiently close to $\mathbf{u}_{\ast}$.
\end{lemma}

\begin{proof}
Since
\begin{eqnarray*}
 \left |\lambda _{k}-\lambda_{\ast}\right |
&=&\min \left ( \frac{\mathcal{A}(\mathbf{u}_{k})\mathbf{u}_{k}}{\mathbf{u}_{k}}- \frac{\mathcal{A}(\mathbf{u}_{\ast})\mathbf{u}_{\ast}}{\mathbf{u}_{\ast}}\right )\le \left \|   \frac{\mathcal{A}(\mathbf{u}_{k})\mathbf{u}_{k}}{\mathbf{u}_{k}}- \frac{\mathcal{A}(\mathbf{u}_{\ast})\mathbf{u}_{\ast}}{\mathbf{u}_{\ast}}\right \| .
\end{eqnarray*}
Since the Fr\'echet derivative of $\frac{\mathcal{A}(\mathbf{u})\mathbf{u}}{\mathbf{u}}$ is continuous in a neighborhood of $\left( \mathbf{u}_{\ast}, \lambda_{\ast}\right)$, we have  $ \left |\lambda _{k}-\lambda_{\ast}\right |  \le
 c_2 \|\mathbf{u}_{k}-\mathbf{u}_{\ast }\|$ for all $\left( \mathbf{u}_{k}, \lambda_{k}\right)$ sufficiently close to $\left( \mathbf{u}_{\ast}, \lambda_{\ast}\right)$.
\end{proof}

We now prove the local quadratic convergence of Algorithm \ref{alg1}.

\begin{theorem}
\label{quadratic}
Assume $\left\{ \mathbf{u}_{k},   \lambda_{k}\right\}$ be generated by NNI. 
Suppose that $\left( \mathbf{u}_{k_0}, \lambda_{k_0}\right) $ is sufficiently close to an eigenpair $\left( \mathbf{u}_{\ast}, \lambda_{\ast}\right) $ with $\mathbf{u}_{\ast} >0$ and $\Vert\mathbf{u}_{\ast}\Vert=1$.
Then $\lambda_{k}$ converges to  $\lambda_{\ast}$ quadratically and $\mathbf{u}_{k}$ converges to  $\mathbf{u}_{\ast}$ quadratically.
\end{theorem}

\begin{proof}
For some $\delta\in (0, \min \mathbf{u}_{\ast})$, there are positive constants $c_1$, $%
c_2 $ and $c_3$ such that
\begin{equation}  \label{eq2.1}
\left\Vert \left[
\begin{array}{c}
\widehat{\mathbf{u}}_{k+1} \\
\widehat{\lambda }_{k+1}%
\end{array}%
\right] -\left[
\begin{array}{c}
\mathbf{u}_{\ast } \\
\lambda_{\ast}%
\end{array}%
\right] \right\Vert \leq c \left\Vert \left[
\begin{array}{c}
\mathbf{u}_{k} \\
{\lambda }_{k}%
\end{array}%
\right] -\left[
\begin{array}{c}
\mathbf{u}_{\ast } \\
\lambda_{\ast}%
\end{array}%
\right] \right\Vert^{2}
\end{equation}
whenever $\left\Vert \left[
\begin{array}{c}
{\mathbf{u}}_{k} \\
{\lambda }_{k}%
\end{array}%
\right] -\left[
\begin{array}{c}
\mathbf{u}_{\ast } \\
\lambda_{\ast}%
\end{array}%
\right] \right\Vert <\delta$,
\begin{equation}
\left |{\lambda }_{k}-\lambda_{\ast}\right | \le c_2 \|\mathbf{u}%
_{k}-\mathbf{u}_{\ast }\|  \label{eq2.2}
\end{equation}
whenever $\left\Vert {\mathbf{u}}_{k} -\mathbf{u}_{\ast } \right\Vert <\delta$, and
\begin{equation}
\left \| F(\widehat{\mathbf{u}}_{k+1},\widehat{\lambda}_{k+1} ) - F(\mathbf{u}_{\ast},\lambda_{\ast} ) \right \| \le c_3 \left\Vert \left[
\begin{array}{c}
\widehat{\mathbf{u}}_{k+1} \\
\widehat{\lambda }_{k+1}%
\end{array}%
\right] -\left[
\begin{array}{c}
\mathbf{x}_{\ast } \\
\lambda_{\ast}%
\end{array}%
\right] \right\Vert  \label{eq2.3}
\end{equation}
whenever $\left\Vert \left[
\begin{array}{c}
\widehat{\mathbf{u}}_{k+1} \\
\widehat{\lambda }_{k+1}%
\end{array}%
\right] -\left[
\begin{array}{c}
\mathbf{u}_{\ast } \\
\lambda_{\ast}%
\end{array}%
\right] \right\Vert <\delta$.
Note that $\mathbf{u}_k > 0$ is guaranteed. Now for all $\epsilon>0$ sufficiently small,assume that 
$\left\Vert \left[
\begin{array}{c}
{\mathbf{u}}_{k} \\
{\lambda }_{k}%
\end{array}%
\right] -\left[
\begin{array}{c}
\mathbf{u}_{\ast } \\
\lambda_{\ast}%
\end{array}%
\right] \right\Vert <\epsilon$ for $k=k_0$. By (\ref{eq2.1}) and (\ref{eq2.2}) we have (with $\epsilon \le \delta$)
\begin{equation*}
\|\widehat{\mathbf{u}}_{k+1} - \mathbf{u}_{\ast } \| 
\le c(1+c_2)^2\|{%
\mathbf{u}}_{k} - \mathbf{x}_{\ast } \|^2 \le  c(1+c_2)^2\epsilon^2.
\end{equation*}

By \eqref{eq2.3} and \eqref{eq2.1} we have (with $\epsilon \le \delta$, $c\epsilon^2 \le \delta$)
\begin{equation*}
\left |\frac{1}{2}\left( 1-\widehat{\mathbf{u}}_{k+1}^{T}\widehat{\mathbf{u}}_{k+1}\right) \right | \le c_3c(1+c_2)^2\|{\mathbf{u}}_{k} -
\mathbf{u}_{\ast } \|^2 \le  c_3c(1+c_2)^2\epsilon^2. 
\end{equation*}
Then $\Vert \widehat{\mathbf{u}}_{k+1} \Vert \ge \frac{1}{2}$ (with $c_3c(1+c_2)^2\epsilon^2 < \frac{3}{8}$).
Now
\begin{equation*}
\left |\|\widehat{\mathbf{u}}_{k+1}\| -1 \right | \le \frac{1}{\left |\|\widehat{\mathbf{u}}_{k+1}\| +1 \right | }2c_3c(1+c_2)^2\|{\mathbf{u}}_{k} -
\mathbf{u}_{\ast } \|^2\le \frac{4}{3}c_3c(1+c_2)^2\|{\mathbf{u}}_{k} -
\mathbf{u}_{\ast } \|^2 . 
\end{equation*}
Then 
\begin{eqnarray*}
\|{\mathbf{u}}_{k+1} - \mathbf{x}_{\ast } \|&=&\| {\mathbf{u}}_{k+1} - 
\widehat{\mathbf{u}}_{k+1} + \widehat{\mathbf{u}}_{k+1} - \mathbf{u}_{\ast }
\|\\
&\le &\| {\mathbf{u}}_{k+1} - \widehat{\mathbf{u}}_{k+1}\|+\| \widehat{%
\mathbf{u}}_{k+1} - \mathbf{u}_{\ast } \| \\
&=& \| \left ({\mathbf{u}}_{k+1} - \|\widehat{\mathbf{u}}_{k+1}\|{\mathbf{u%
}}_{k+1}\right ) \|+\| \widehat{\mathbf{u}}_{k+1} - \mathbf{u}_{\ast } \|
\\
&=&\left | \|\widehat{\mathbf{u}}_{k+1}\|-1\right | +\| \widehat{\mathbf{u}%
}_{k+1} - \mathbf{u}_{\ast } \| \\
&\le & \left (1+\frac{4}{3}c_3\right ) c(1+c_2)^2\|{\mathbf{u}}_{k} -
\mathbf{u}_{\ast } \|^2.
\end{eqnarray*}
For $\epsilon$ with $(1+\frac{4}{3}c_3) c(1+c_2)^2 \epsilon \le
\frac{1}{1+c_2}$, we have  $\|{\mathbf{u}}_{k+1} - \mathbf{u}_{\ast } \| < 
\frac{1}{1+c_2}\epsilon$ and thus  $\left |{{\lambda }}%
_{k+1}-\lambda_{\ast}\right | \le c_2 \|\mathbf{u}_{k+1}-\mathbf{u}_{\ast
}\| < \frac{c_2}{1+c_2}\epsilon$. Therefore, $\left\Vert \left[
\begin{array}{c}
{\mathbf{u}}_{k+1} \\
{\lambda }_{k+1}%
\end{array}%
\right] -\left[
\begin{array}{c}
\mathbf{u}_{\ast } \\
\lambda_{\ast}%
\end{array}%
\right] \right\Vert = \|{\mathbf{u}}_{k+1} - \mathbf{u}_{\ast } \|+
\left |{\lambda }_{k+1}-\lambda_{\ast}\right | < \epsilon$. We can then
repeat the above process to get $\|{\mathbf{u}}_{k+1} - \mathbf{u}_{\ast }
\| \le d \|{\mathbf{u}}_{k} - \mathbf{u}_{\ast } \|^2$ for all $k\ge k_0$
and $d = \left (1+\frac{4}{3}c_3\right ) c(1+c_2)^2$. Thus $\mathbf{u}_{k}$
converges to $\mathbf{u}_{\ast}$ quadratically and then ${\lambda }%
_{k}$ converges to $\lambda_{\ast}$ quadratically by \eqref{eq2.2}.

\end{proof}

\section{Numerical experiments}

\label{sec:exp} In this section, we present some numerical results to support
our theory for NNI and illustrate its effectiveness. All numerical tests were performed on 4.2GHz quad-core Intel Core i7  with $32$ GB memory using Matlab R$2018b $ with machine precision $\varepsilon =2.22\times 10^{-16}$ under the
macOS High Sierra. Throughout the experiments, the initial vector is $\mathbf{u}_{0}=\frac{1}{\sqrt{n}}[1,\ldots ,1]^{T}\in \mathbb{R}^{n}$. 
In the experiments, the stopping criterion for NNI is the relative residual
\begin{equation*}
\frac{\Vert \mathcal{A}(\mathbf{u}_k)-\lambda_k\mathbf{u}_k\Vert} {%
(\|\mathcal{A}(\mathbf{u}_k)\|_1\|\mathcal{A}(\mathbf{u}_k)\|_{\infty})^{1/2}}  \le 10^{-12},
\end{equation*}
where we use the cheaply computable $(\|\mathcal{A}(\mathbf{u}_k)\|_1\|\mathcal{A}(\mathbf{u}_k)\|_{\infty})^{1/2}$
to estimate the 2-norm $\|\mathcal{A}(\mathbf{u}_k)\|$, which is more reasonable
than the individual $\|\mathcal{A}(\mathbf{u}_k)\|_1$ or $\|\mathcal{A}(\mathbf{u}_k)\|_{\infty}$ with $\|\cdot\|_{\infty}$
the infinity norm of a matrix.

\begin{example}
\label{exp:FEM} Consider the finite-difference discretization of the nonlinear eigenvalue problem (\ref{eq:NSLE}) with Dirichlet boundary conditions on $[0,1]\times [0,1]$, i.e.,
\begin{eqnarray*}
A\mathbf{u}+\Gamma \mathrm{diag} \left (\mathbf{e}-\frac{\mathbf{e}}{\mathbf{%
a}+\mathbf{u}^{[2]}}\right ) \mathbf{u}=\lambda \mathbf{u}, 
\end{eqnarray*}%
where $A  \in  \mathbb{R}^{n \times n}$ is a negative 2D Laplacian matrix with Dirichlet boundary conditions.
\end{example}

 For Example~\ref{exp:FEM}, Figure~\ref{fig:rand} depicts how the relative residual evolves versus the
number of iterations for NNI.  It shows that NNI uses $8$ iterations to achieve the required accuracy, clearly indicating its quadratic convergence.
\begin{figure}[hbtp]
\centering
\epsfig{file=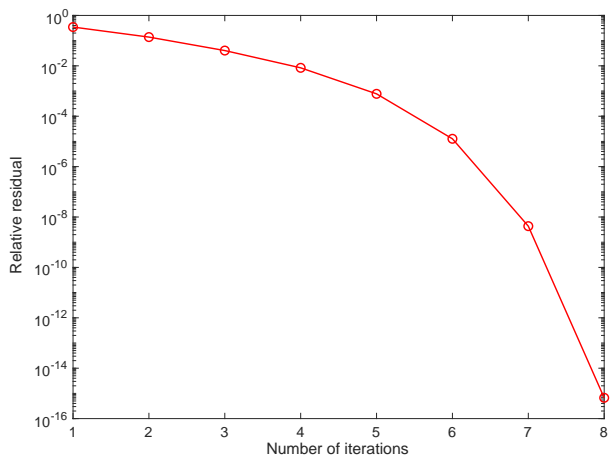,width=.7\textwidth}
\caption{The relative residual versus the number of iterations for $n=100, \Gamma=10$ and $1>\mathbf{a}>0$.}
\label{fig:rand}
\end{figure}
Figure~\ref{fig:gamma} shows that the magnitude of parameter $\Gamma$ affects the total number of iterations to achieve convergence. As we see, NNI requires more iterations to achieve convergence for lager parameter $\Gamma$.
\begin{figure}[hbtp]
\centering
\epsfig{file=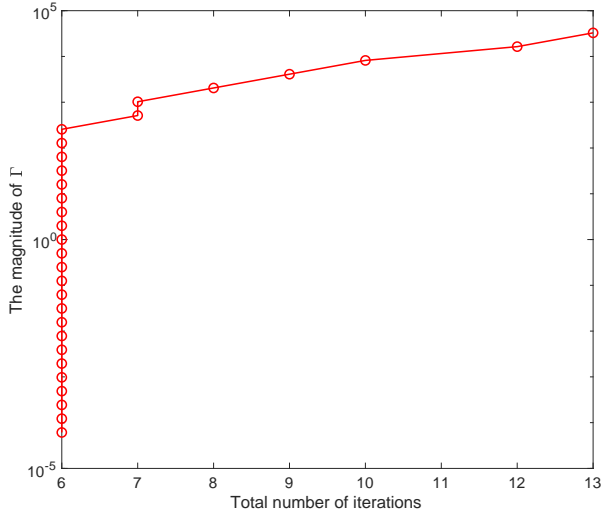,width=.7\textwidth}
\caption{The magnitude of $\Gamma$ versus the total number of iterations for $n=2500$ and $\mathbf{a}\ge1$.}
\label{fig:gamma}
\end{figure}

Table~\ref{tbl:outer} reports the results obtained by NNI.
In the table, $n$ specifies the dimension, $\mathbf{a}$ is a
parameter to adjust the diagonal elements of $\mathrm{diag} \left (\mathbf{e}-\frac{\mathbf{e}}{\mathbf{a}+\mathbf{u}^{[2]}}\right )$. \textquotedblleft $\mathbf{a}\ge1$\textquotedblright  denotes that each element of $\mathbf{a}$ is larger than $1$, \textquotedblleft $1>\mathbf{a}>0$\textquotedblright  denotes that each element of $\mathbf{a}$ is between $0$ and $1$, \textquotedblleft $\mathbf{a}>0$\textquotedblright  denotes that each element of $\mathbf{a}$ is larger than $0$.
\textquotedblleft Iter\textquotedblright denotes the number of iterations
to achieve convergence, \textquotedblleft Residual\textquotedblright
denotes the relative residual when NNI is terminated. From the table, we see that the number of iterations for NNI is at most $23$, clearly indicating its rapid convergence. For this example, $\mathbf{h}_{k}(\theta_k)>0$ holds with $\theta_k = 1$ for each iteration of NNI and the halving procedure is not used. These results indicate that our theory of NNI can be conservative.

\begin{table}[htbp]
\caption{Numerical results for NNI}
\label{tbl:outer}\centering
\begin{tabular}{rrlrll}
\hline
Parameters &  &  & NNI      \\
\cline{1-2}\cline{4-5}
$n$ & $\mathbf{a} $ &  & Iter & 	Residual     \\
\hline
$2500$ & $\mathbf{a}\ge1$ &  & $6$ & 2.52e-16     \\
$10000$ & $\mathbf{a}\ge1$ &  & $6$ & 2.38e-16     \\
$40000$ & $\mathbf{a}\ge1$ &  & $6$ & 2.79e-16     \\  \hline

$2500$ & $1>\mathbf{a}>0$ &  & $13$ & 2.15e-16     \\
$10000$ & $1>\mathbf{a}>0$ &  & $16$ & 2.25e-16     \\
$40000$ & $1>\mathbf{a}>0$ &  & $23$ & 2.50e-16     \\ \hline

$2500$ & $\mathbf{a}>0$ &  & $13$ & 1.98e-16     \\
$10000$ & $\mathbf{a}>0$ &  & $15$ & 1.34e-15     \\ 
$40000$ & $\mathbf{a}>0$ &  & $21$ & 7.04e-16     \\ \hline
\end{tabular}%
\end{table}

\section{Conclusion}
In this paper, we are concerned with the nonlinear algebraic eigenvalue problem (NAEP) generated by the discretization of the saturable nonlinear Schr\"odinger equation. Based on the idea of Noda's iteration and Newton's method, we have proposed an effective method for computing the positive eigenvectors of NAEP, called Newton--Noda iteration. It involves the selection of a positive parameter $\theta_k$ in the $k$th iteration. We have presented a halving procedure to determine the parameters $\theta_k$, starting with $\theta_k=1$ for each iteration, such that the sequence approximating target eigenvalue $\lambda_{\ast}$ is strictly monotonic increasing and bounded, and thus its global convergence is guaranteed for any initial positive unit vector. Additionally, another advantage of the presented method is its local convergence speed. We have shown that the parameter $\theta_k$ is chosen eventually equal to $1$ and locally quadratic convergence is achieved. The numerical experiments have indicated that the halving procedure will often return $\theta_k=1$ (i.e., no halving is actually used) for each $k$, and near convergence the halving procedure will always return $\theta_k=1$. These results confirm our theory and demonstrate that our theoretical results can be realistic and pronounced.
 

This iterative method has several nice features:  Structure Preserving--It maintains positivity in its computation of positive ground state vectors, and its convergence is global and quadratic. Easy-to-implement --The structure of the new algorithm is still very simple, although its convergence analysis is rather involved for nonlinear algebraic eigenvalue problems. On the other hand, it gives an alternative approach to approximate the solution of the nonlinear Schrödinger equation by constructing a sequence. This is precisely the way we use to prove the existence of solutions of the discrete nonlinear Schr\"odinger equation.

\section*{Acknowledgements}
I am very grateful to the Ministry of Science and Technology in Taiwan for funding this research, and I would like to thank Chun-Hua Guo, Wen-Wei Lin and Tai-Chia Lin for their valuable comments on this paper.

\end{document}